\documentclass[12pt]{amsart}

\usepackage{amsmath}
\usepackage{amssymb}
\usepackage{enumerate}

\renewcommand{\pmod}[1]{\,(\textup{mod}\,#1)}
\numberwithin{equation}{section}
 \theoremstyle{plain}
\newtheorem{theorem}{Theorem}[section]

\newtheorem{lemma}[theorem]{Lemma}
\newtheorem{corollary}[theorem]{Corollary}

\newcommand{\wt}{\widetilde{\sigma}}
\newcommand{\mP}{\mathcal{P}}
\newcommand{\mQ}{\mathcal{Q}}
\newcommand{\mE}{\mathcal{E}}

\begin{document}

\title{Eisenstein series associated with $\Gamma_0(2)$}
\author{Heekyoung Hahn}
\address{Department of Mathematics, University of Rochester, Rochester, NY 14627 USA}
\email{hahn@math.rochester.edu}
\begin{abstract}
In this paper, we define the normalized Eisenstein series $\mP$, $e$, and $\mQ$ associated with $\Gamma_0(2),$ and derive three differntial equations satisfied by them from some trigonometirc identities. By using these three formulas, we define a differential equation depending on the weights of modular forms on $\Gamma_0(2)$ and then construct its modular solutions by using orthogonal polynomials and Gaussian hypergeometric series. We also construct a certain class of infinite series connected with the triangular numbers. Finally, we derive a combinatorial identity from a formula involving the triangular numbers.
\end{abstract}
\maketitle

\section{Introduction}
In his notation \cite{Ramanujan-arith}, Ramanujan's three primary
Eisenstein series are defined for $|q|<1$ by
{\allowdisplaybreaks\begin{align}
P&:=P(q)=1-24\Phi_{0,1}(q)=1-24\sum_{n=1}^{\infty}\frac{nq^n}{1-q^n},\label{P}\\
Q&:=Q(q)=1+240\Phi_{0,3}(q)=1+240\sum_{n=1}^{\infty}\frac{n^3q^n}{1-q^n},\label{Q}\\
R&:=R(q)=1-504\Phi_{0,5}(q)=1-504\sum_{n=1}^{\infty}\frac{n^5q^n}{1-q^n},\label{R}
\end{align}}
where
$$\Phi_{r,s}:=\Phi_{r,s}(q)=\sum_{m=1}^{\infty}
\sum_{n=1}^{\infty} m^r n^s q^{mn}$$ for integers $r, s \geq 0$.

In more contemporary notation, the normalized Eisenstein series on $SL_2(\mathbb{Z})$ are defined, for each even integer $k \geq 4$, by
$$E_k(z)=\frac{1}{2}\sum (cz+d)^{-k},$$
where the summation is over all coprime pairs of integers $c$ and
$d$, and $\text{Im }z > 0$. Then it is known that $E_k(z)$ has the
Fourier expansion \cite{Koblitz}
\begin{equation}
E_k:=E_k(z)=1-\frac{2k}{B_k} \sum_{n=1}^{\infty}\sigma_{k-1}(n)q^n,
\end{equation}
where $q=e^{2 \pi i z}$, $B_k$ is the $k$th Bernoulli number and
\begin{equation}\label{sigma}
\sigma_k(n):=\sum_{d|n}d^k.
\end{equation} As usual, we set
$\sigma_1(n)=\sigma(n)$ and $\sigma_k(n)=0$ if $n \notin
\mathbb{N}$. Note that $E_4(z)=Q(q)$ and $E_6(z)=R(q)$ are holomorphic modular forms on $SL_2(\mathbb{Z})$ of weights $4$ and $6$, respectively \cite[p. 109]{Koblitz}. It is well--known that $E_2(z)=P(q)$ is not a modular form of weight $2$ \cite[p. 12]{Koblitz}, called a {\it quasi--modular} form. The Eisenstein series \eqref{P}--\eqref{R} satisfy the differential equations \cite[eq. (30)]{Ramanujan-arith}, \cite[p. 142]{Ramanujan-collect}
{\allowdisplaybreaks \begin{align}
q\frac{dP}{dq}&=\frac{P^2-Q}{12}, \label{dP}\\
q\frac{dQ}{dq}&=\frac{PQ-R}{3}, \label{dQ}\\
q\frac{dR}{dq}&=\frac{PR-Q^2}{2}. \label{dR}
\end{align}}
By analogy with Ramanujan's functions $\Phi_{r,s}$, V. Ramamani, in her paper \cite[pp. 279--286]{Ramamani-paper}, defined $\Phi_{r,s}$ for integers $r,s \geq 0$, by
\begin{equation}\label{Psi}
\Psi_{r,s}:=\Psi_{r,s}(q)=\sum_{m=1}^{\infty}\sum_{n=1}^{\infty}(-1)^{n-1}m^r
n^s q^{mn}.
\end{equation}
In contrast to $\Phi_{r,s}$, the functions $\Psi_{r,s}$ are not symmetric in $r$ and $s$. In connection with $\Psi_{r,s}$, let us define three functions $\mP$, $e$, and $\mQ$ by
{\allowdisplaybreaks\begin{align}
\mP&:=\mP(q)=1+8\Psi_{0,1}(q)=1+8\sum_{n=1}^{\infty}\frac{(-1)^{n-1}nq^n}{1-q^n}, \label{mp}\\
e&:=e(q)=1+24\Psi_{1,0}(q)=1+24\sum_{n=1}^{\infty}\frac{nq^n}{1+q^n},\label{e}\\
\mQ&:=\mQ(q)=1-16\Psi_{0,3}(q)=1-16\sum_{n=1}^{\infty}\frac{(-1)^{n-1}n^3q^n}{1-q^n}.\label{mq}
\end{align}}
Using the theory of the elliptic functions, Ramamani \cite{Ramamani-thesis} proved that when $r+s$ is odd, $\Psi_{r,s}$ can be expressed as a polynomial in $\mP$, $e$, and $\mQ$. We remark that for $r+s$ odd, the function $\Psi_{r,s}$ is
related to the normalized Eisenstein series on $\Gamma_0(2)$, where the modular subgroup $\Gamma_0(2)$ is defined by
\begin{equation}\label{gamma2}
\Gamma_0(2):=\left\{\gamma\in SL_2(\mathbb{Z})\,\Big|\, \gamma\equiv \begin{pmatrix} * & * \\0 & *\end{pmatrix} \pmod{2}\right\}.
\end{equation}
The normalized Eisenstein series
associated with $\Gamma_0(2)$ are defined, for even integer $k
\geq 4$, by
\begin{equation}\label{eisen}
\mE_k:=\mE_k(z)=1-\frac{2k}{(1-2^k)B_k} \sum_{n=1}^{\infty}
\frac{(-1)^{n-1} n^{k-1}q^n}{1-q^n}.
\end{equation}
Then the series $\mE_k(z)$ are modular forms of weight $k$ on $\Gamma_0(2)$ which vanish at the cusp zero \cite[Theorem 1.1]{Boylan}. It is clear that $\mE_4(z)=\mQ(q)$ is the relevant modular form of weight $4$ on $\Gamma_0(2)$. When $k=2$, it turns out that $\mE_2(z)=\mP(q)$ is not a modular form on this group (see \eqref{rel-mE2}), but it plays important roles in the theory of modular forms of level $2$. Note that the function $e(q)$ is indeed the modular form of weight $2$ on $\Gamma_0(2)$ \cite[Lemma 3.3]{BBY}.

In Section 2, we derive relations for $\Psi_{r,s}$, for odd $r+s$, from trigonometric identities \cite[eqs. (17), (18)]{Ramanujan-arith} and then, in the same manner as \eqref{dP}--\eqref{dR} are proved, we obtain the differential equations
{\allowdisplaybreaks\begin{align}
q\frac{d\mP}{dq}&=\frac{\mP^2-\mQ}{4}, \label{dmp}\\
q\frac{de}{dq}&=\frac{e\mP-\mQ}{2}, \label{de}\\
q\frac{d\mQ}{dq}&=\mP\mQ-e\mQ. \label{dmq}
\end{align}}The proofs will be given in Theorem \ref{TDmpmq} and Theorem \ref{TDe}. At the end of Section 2, we mention an alternative proof of these formulas using the theory of modular forms. In Section 3, by using \eqref{dmp}--\eqref{dmq}, we define a differential equation depending on the weight $k$ of modular forms on $\Gamma_0(2)$ and then construct its modular solutions by using orthogonal polynomials. We also find the hypergeometric structure for the solutions of this differential equation. In section 4, we construct a certain class of infinite series connected with the triangular numbers. Finally, we derive a combinatorial interpretation from one of formulas which we construct in Section 4.

\section{Differential equations for $\mP$, $e$, and $\mQ$}

Ramamani \cite{Ramamani-thesis} proved that for odd $s \geq 3$, $\Psi_{0,s}$ can be expressed as a polynomial in $e$ and $\mQ$ by using the theory of elliptic functions. We can observe this by comparing the coefficients of $\theta^n$ in the trigonometric identity \cite[eq. (18)]{Ramanujan-arith}
{\allowdisplaybreaks\begin{align} \Big( \frac{1}{8} \cot
^2{\frac{\theta}{2}}+\frac{1}{12} \Big)^{2}
 &+\frac{1}{12} \sum_{n=1}^{\infty} \frac{n^3q^n}{1-q^n}(5+\cos {n \theta} )\label{Tri}\\
&=\Big(\frac{1}{8}\cot ^2{\frac{\theta}{2}}+\frac{1}{12}+\sum_{n=1}^{\infty} \frac{nq^n}{1-q^n}(1-\cos {n \theta} ) \Big)^{2}.\nonumber
\end{align}}
After replacing $\theta$ by $\pi+\theta$ in \eqref{Tri}, let us expand $\tan \theta$ and $\cos n \theta$ in their Taylor series about $0$ for each $n=1,2, \dots$. We therefore find that
{\allowdisplaybreaks
\begin{align}
\frac{1}{144}+\frac{\theta^2}{192}&+\frac{17\theta^4}{9216}+\cdots
+\frac{1}{12}\Bigg\{\Big(4\frac{q}{1-q}+6\frac{2^3q^2}{1-q^2}+4\frac{3^3q^3}{1-q^3}+\cdots\Big)\nonumber\\
&\hspace{1.5in}+\frac{1}{2!}\Big(\frac{q}{1-q}-\frac{2^5q^2}{1-q^2}+\frac{3^5q^3}{1-q^3}-\cdots \Big)\theta^2\nonumber\\
&\hspace{1.5in}-\frac{1}{4!}\Big(\frac{q}{1-q}-\frac{2^7q^2}{1-q^2}+\frac{3^7q^3}{1-q^3}-\cdots
\Big)\theta^4+\cdots \Bigg\}\nonumber\\
&=\Bigg\{\frac{1}{12}+2\Big(\frac{q}{1-q}+\frac{3q^3}{1-q^3}+\frac{5q^5}{1-q^5}+\cdots\Big)\label{bothsides}\\
&\hspace{0.8in}+\frac{1}{2!}\Big(\frac{1}{16}-\frac{q}{1-q}+\frac{2^3q^2}{1-q^2}-\frac{3^3q^3}{1-q^3}+\cdots\Big)
\theta^2 \nonumber\\
&\hspace{0.8in}+\frac{1}{4!}\Big(\frac{1}{8}+\frac{q}{1-q}-\frac{2^5q^2}{1-q^2}+\frac{3^5q^3}{1-q^3}-\cdots
\Big) \theta^4+\cdots \Bigg\}^2.\nonumber
\end{align}}
Before going further, we need to introduce an alternative
representation for $e(q)$. By the elementary fact
$$\frac{x}{1+x}=\frac{x}{1-x}-\frac{2x^2}{1-x^2},$$ we find that
{\allowdisplaybreaks\begin{align}
e(q)&=1+24\sum_{n=1}^{\infty}\frac{nq^n}{1+q^n}\nonumber\\&=1+24\sum_{n=1}^{\infty}\frac{nq^n}{1-q^n}-24\sum_{n=1}^{\infty}\frac{2nq^{2n}}{1-q^{2n}}\nonumber\\&=1+24\sum_{n=1}^{\infty}\frac{(2n-1)q^{2n-1}}{1-q^{2n-1}}\label{alt-e}.
\end{align}}
So we can rewrite \eqref{bothsides} as
{\allowdisplaybreaks\begin{align}
\frac{1}{144}&+\frac{1}{6}\Bigg(2\frac{q}{1-q}+3\frac{2^3q^2}{1-q^2}+2\frac{3^3q^3}{1-q^3}+\cdots\Bigg)
+\Bigg(\frac{1}{192}+\frac{\Psi_{0,5}}{12 \cdot2!}\Bigg)\theta^2\nonumber\\
&\hspace{.4in}+\Bigg(\frac{17}{9216}-\frac{\Psi_{0,7}}{12\cdot4!}\Bigg)\theta^4
+\Bigg(\frac{31}{69120}+\frac{\Psi_{0,9}}{12\cdot6!}\Bigg)\theta^6+\cdots\nonumber\\
&=\Bigg\{\frac{1}{12}+\frac{e-1}{12}+\frac{\mQ}{16\cdot2!}\theta^2+\frac{1+8\Psi_{0,5}}{8\cdot4!}\theta^4+\cdots
\Bigg\}^2\nonumber\\
&=\frac{e^2}{144}+\frac{e\mQ}{192}\theta^2+\Bigg(\frac{\mQ^2}{1024}
+\frac{e(1+8\Psi_{0,5})}{1152}\Bigg)\theta^4\label{finalsides}\\
&\hspace{.4in}+\Bigg(\frac{e(17-32\Psi_{0,7})}{139240}+\frac{\mQ(1+8\Psi_{0,5})}{3072}\Bigg)\theta^6+\cdots.\nonumber
\end{align}}
So if we compare the coefficients of $\theta^2$ on both sides of \eqref{finalsides}, then we have 
\begin{equation}\label{psi05}
1+8\Psi_{0,5}=e\mQ.\end{equation} 
Similarly, equating coefficients of $\theta^4$ and using \eqref{psi05}, we obtain the identity
$$17-32\Psi_{0,7}=9\mQ^2+8e(1+8\Psi_{0,5})=9\mQ^2+8e^2\mQ.$$
Successively comparing the coefficients of $\theta^n,~n=6,10,12,
\dots$ on both sides, we easily obtain the following theorem.

\begin{theorem}\label{mainlemma}
For even  integer $k \geq 4$,
\begin{equation}\label{generaleisen}
\mathcal{E}_k=\sum_{\substack{2m+4n=k\\m \geq 0, n \geq 1}}
\alpha_{m,n}e^m \mQ^n,
\end{equation}
where $\alpha_{m,n}$ are constants.
\end{theorem}
The first few examples of Theorem \ref{mainlemma} are the
relations contained in the following Table I. {\allowdisplaybreaks
\begin{align}
1-16\Psi_{0,3}=&\mQ,\\
1+8\Psi_{0,5}=&e\mQ, \label{table5}\\
17-32\Psi_{0,7}=&8e^2\mQ+9\mQ^2,\label{table7}\\
31+8\Psi_{0,9}=&4e^3\mQ+27e\mQ^2,\label{table9} \\
691-16\Psi_{0,11}=&16e^4\mQ+486e^2\mQ^2+189\mQ^3,\\
5461+8\Psi_{0,13}=&16e^5\mQ+2016e^3\mQ^2+3429e\mQ^3,\\
929569-64\Psi_{0,15}=&256e^6\mQ+130464e^4\mQ^2+667872e^2\mQ^3
+130977\mQ^4.
\end{align}}
\begin{center}
\textbf{Table I}
\end{center}

Now, we will give a detailed proof of the differential equations \eqref{dmp}--\eqref{dmq}.

\begin{theorem}\label{TDmpmq}
If $\mP$ and $\mQ$ are defined by \eqref{mp} and \eqref{mq}, respectively, then $\mP$ and $\mQ$ satisfy the differential equations \eqref{dmp} and \eqref{dmq}, respectively.
\end{theorem}

\begin{proof}
Recall the identity \cite[eq. (17)]{Ramanujan-arith}
{\allowdisplaybreaks\begin{align}
\Big(\frac{1}{4}\cot{\frac{\theta}{2}}+\sum_{n=1}^{\infty}\frac{q^n}{1-q^n}\sin{n\theta}\Big)^{2}
=&\Big(\frac{1}{4}\cot{\frac{\theta}{2}}\Big)^{2}+\sum_{n=1}^{\infty} \frac{q^n \cos{n\theta}}{(1-q^n)^2}\nonumber \\
 &+\frac{1}{2}\sum_{n=1}^{\infty}\frac{nq^n}{1-q^n}(1-\cos{n\theta}).\label{Trigo}
\end{align}}Replacing $\theta$ by $\pi +\theta$ in \eqref{Trigo} and expanding $\sin n\theta$ and $\cos n\theta$ in Maclaurin series, we obtain the Taylor series expansion at $0$,
{\allowdisplaybreaks\begin{align} \Big( \frac{1}{8}\mP
\theta+&\frac{1}{16}\mQ \frac{ \theta ^3}{3!} + \Big( \frac{1}{8}
+\Psi_{0,5} \Big)\frac{\theta ^5}{5!}
+\Big( \frac{17}{32} -\Psi_{0,7} \Big)\frac{\theta ^7}{7!}+ \cdot \cdot \cdot \Big)^{2} \nonumber \\
=&\frac{1}{2} \sum_{n=1}^{\infty}
\frac{nq^n}{1-q^n}-\sum_{n=1}^{\infty}
(-1)^{n-1}\frac{q^n}{(1-q^n)^2} +\frac{1}{2}\Psi_{0,1} \label{ETrigo}\\
&+ \Big( \frac{1}{32}+\Psi_{1,2}-\frac{1}{2}\Psi_{0,3} \Big)
\frac{\theta ^2}{2!}
+ \Big( \frac{1}{16}-\Psi_{1,4}+\frac{1}{2}\Psi_{0,5} \Big)
\frac{\theta ^4}{4!}+\cdots . \nonumber
\end{align}}
If we compare the coefficient of $\theta^2$ on both sides of
\eqref{ETrigo}, then we deduce that
\begin{equation}
\mP^2=1+32\Psi_{1,2}-16\Psi_{0,3}=\mQ+32\Psi_{1,2}.\label{e-dmp}
\end{equation}
By the definition of $\Psi_{0,s}$, $s \geq 1$, it is clear that
\begin{equation}\label{dpsirs}
q\frac{d\Psi_{0,s}}{dq}=\Psi_{1,s+1} \quad \text{and} \quad
q\frac{d\Psi_{r,0}}{dq}=\Psi_{r+1,1}.
\end{equation}
So by \eqref{e-dmp}, we obtain
\begin{equation}
q\frac{d\mP}{dq}=8q\frac{d\Psi_{0,1}}{dq}=8\Psi_{1,2}=\frac{\mP^2-\mQ}{4},
\end{equation}
which is the desired identity \eqref{dmp}.

Similarly, by comparing the coefficients of $\theta^4$ from
\eqref{ETrigo}, we can find that
\begin{equation}\label{Psi14}
\frac{\mP\mQ}{16}=\frac{1}{16}-\Psi_{1,4}+\frac{1}{2}\Psi_{0,5}.
\end{equation}
Hence, from \eqref{Psi14}, we derive that
$$q\frac{d\mQ}{dq}=-16q\frac{d\Psi_{0,3}}{dq}=-16\Psi_{1,4}=\mP\mQ-(1+8\Psi_{0,5})=\mP\mQ-e\mQ,$$
from \eqref{table5}.
\end{proof}

To find a differential equation for $e(q)$, we need another trigonometric identity proved by Ramamani \cite[eq. (1.5)]{Ramamani-paper}.

\begin{lemma}
{\allowdisplaybreaks\begin{align} \Big( \frac{1}{4}\cot
\frac{\theta}{2}&+\sum_{n=1}^{\infty}\frac{q^n}{1-q^n}\sin n\theta
\Big)^{3}=\Big( \frac{\cot \theta/2}{4}
\Big)^{3}-\frac{3}{2}\sum_{n=1}^{\infty}\frac{q^n}{(1-q^n)^3}\sin
n \theta \nonumber \\
&+\frac{3}{4}\sum_{n=1}^{\infty}\frac{(n+1)q^n}{(1-q^n)^2}\sin n
\theta-\frac{1}{16}\sum_{n=1}^{\infty}\frac{(2n^2+1)q^n}{1-q^n}\sin{n\theta} \label{RamamaniTri}\\
&+\frac{3}{8}\cot\frac{\theta}{2}\sum_{n=1}^{\infty}\frac{nq^n}{1-q^n}+
\frac{3}{2}\Big( \sum_{n=1}^{\infty}\frac{q^n}{1-q^n}\sin n\theta
\Big) \Big(\sum_{n=1}^{\infty} \frac{nq^n}{1-q^n} \Big). \nonumber
\end{align}}
\end{lemma}

\begin{theorem}\label{TDe}
If $e(q)$ is defined by \eqref{e}, then $e(q)$ satisfies differential
equation \eqref{de}.
\end{theorem}

Note that for the proof of Theorem \ref{TDe}, Ramamani \cite[p. 116]{Ramamani-thesis} briefly mentioned that the equation \eqref{de} can be obtained by comparing the coefficient of $\theta$ in the Taylor expansions around $0$ in \eqref{RamamaniTri}, after replacing $\theta$ by $\pi +\theta$. We will show this here in detail. We first need the following simple, but useful fact.

\begin{lemma} 
Let $P(q)$, $\mP(q)$, and $e(q)$ be as in \eqref{P}, \eqref{mp}, and \eqref{e}, then
\begin{equation}\label{3mp2e}
P(q)=3\mP(q)-2e(q).
\end{equation}
\end{lemma}
\begin{proof}
By \eqref{alt-e}, we obtain {\allowdisplaybreaks\begin{align*}
3\mP(q)-2e(q)=&3\Big(1+8\sum_{n=1}^{\infty}\frac{(2n-1)q^{2n-1}}{1-q^{2n-1}}
-8\sum_{n=1}^{\infty}\frac{2nq^{2n}}{1-q^{2n}}\Big)\\
&-2\Big(1+24\sum_{n=1}^{\infty}\frac{(2n-1)q^{2n-1}}{1-q^{2n-1}}\Big)\\
=&1-24\sum_{n=1}^{\infty}\frac{(2n-1)q^{2n-1}}{1-q^{2n-1}}-24\sum_{n=1}^{\infty}\frac{2nq^{2n}}{1-q^{2n}}\\
=&1-24\sum_{n=1}^{\infty}\frac{nq^n}{1-q^n}.\end{align*}}
\end{proof}
Now to simplify the infinite series on the right hand side of
\eqref{RamamaniTri}, a new series representation for $\Psi_{2,1}$
shown below is necessary.

\begin{lemma}
If $\Psi_{r,s}(q)$ is defined by \eqref{Psi}, then
\begin{equation}\label{NewPsi21}
\Psi_{2,1}(q):=\sum_{n=1}^{\infty}\frac{(-1)^{n-1}nq^n(1+q^n)}{(1-q^n)^3}.
\end{equation}
\end{lemma}
\begin{proof}
By the definition of $\Psi_{r,s}(q)$, we easily derive that
\begin{equation}
\Psi_{1,s}(q)=\sum_{n=1}^{\infty}\frac{(-1)^{n-1}n^sq^n}{(1-q^n)^2}.\label{Psi1s}
\end{equation}
Setting $s=0$ in \eqref{Psi1s}, we find that
$$\sum_{n=1}^{\infty}\frac{(-1)^{n-1}q^n}{(1-q^n)^2}=\Psi_{1,0}(q).\label{Psi10}$$
Differentiate both sides of the above equality and then multiply
by $q$. By \eqref{dpsirs}, in the case $r=1$, we complete the
proof.
\end{proof}

\begin{proof}[Proof of Theorem \ref{TDe}.] 
After replacing $\theta$ by $\pi +\theta$ in \eqref{RamamaniTri}
and comparing the coefficients of $\theta$ for the Taylor
expansions at $0$, we can find that
{\allowdisplaybreaks\begin{align}
0=&-\frac{3}{2}\sum_{n=1}^{\infty}\frac{(-1)^{n-1}nq^n}{(1-q^n)^3}+\frac{3}{4}\sum_{n=1}^{\infty}
\frac{(-1)^{n-1}nq^n}{(1-q^n)^2}+\frac{3}{4}\sum_{n=1}^{\infty}\frac{(-1)^{n-1}n^2q^n}{(1-q^n)^2}\\
&-\frac{1}{16}\sum_{n=1}^{\infty}\frac{(-1)^{n-1}(2n^3+n)q^n}{1-q^n}
+\left(\frac{3}{16}\sum_{n=1}^{\infty}\frac{nq^n}{1-q^n}\right)\cdot\left(1+8\sum_{n=1}^{\infty}
\frac{(-1)^{n-1}nq^n}{1-q^n}\right).\nonumber
\end{align}}
For convenience, set {\allowdisplaybreaks\begin{align}
S_1:=&-\frac{3}{2}\sum_{n=1}^{\infty}
\frac{(-1)^{n-1}nq^n}{(1-q^n)^3}
   +\frac{3}{4}\sum_{n=1}^{\infty} \frac{(-1)^{n-1}nq^n}{(1-q^n)^2}, \\
S_2:=&\frac{3}{4}\sum_{n=1}^{\infty} \frac{(-1)^{n-1}n^2q^n}{(1-q^n)^2},\\
S_3:=&-\frac{1}{16}\sum_{n=1}^{\infty} \frac{(-1)^{n-1}(2n^3+n)q^n}{1-q^n},\\
S_4:=&\left(\frac{3}{16}\sum_{n=1}^{\infty}\frac{nq^n}{1-q^n}\right)\cdot
\left(1+8\sum_{n=1}^{\infty}\frac{(-1)^{n-1}nq^n}{1-q^n}\right).
\end{align}}
By \eqref{NewPsi21} and a simple calculation, we obtain $$S_1=-\frac{3}{4}\Psi_{2,1}.$$ For $S_2$, we have
\begin{equation}
S_2=\frac{3}{4}q\frac{d}{dq}
\left(\sum_{n=1}^{\infty}\frac{(-1)^{n-1}nq^n}{1-q^n} \right)
=\frac{3}{32} \left( q\frac{d\mP}{dq} \right)=\frac{3}{128}
(\mP^2-\mQ),
\end{equation}
where the last equality comes from \eqref{dmp}. By the definitions
of $\mP$ and $\mQ$ given in \eqref{mp} and \eqref{mq},
respectively, $$S_3=\frac{\mQ-\mP}{128}.$$ Finally, use
\eqref{3mp2e} to deduce that
$$S_4=\frac{3(1-P)\mP}{128}=\frac{(1-3\mP-2e)\mP}{128}.$$ It
follows that
\begin{align*}
0=&S_1+S_2+S_3+S_4 \\
 =&-\frac{3}{4}\Psi_{2,1}+\frac{3(\mP^2-\mQ)}{128}+\frac{\mQ-\mP}{128}+\frac{\mP-3\mP^2-2e\mP}{128},
\end{align*}
and hence we obtain
$$\Psi_{2,1}=\frac{e\mP-\mQ}{48}.$$ By using $$q\frac{de}{dq}=24\Psi_{2,1},$$ we complete the proof.
\end{proof}

\noindent \textbf{Remark.} It is possible to derive \eqref{dmp}--\eqref{dmq} from the analogous formulas in level one.

Let $\Theta:=q\frac{d}{dq}=\frac{1}{2\pi i}\frac{d}{dz},$ and use the notations $\mE_2$, $e_2$, and $\mE_4$, rather than $\mP$, $e$, and $\mQ$,
respectively, because we want to focus on their weights. Then the
differential equations \eqref{dmp}--\eqref{dmq} are, in these
notations, {\allowdisplaybreaks\begin{align}
\Theta \mE_2&=\frac{\mE_2^2-\mE_4}{4}, \label{dmp-w}\\
\Theta e_2&=\frac{e_2\mE_2-\mE_4}{2},\label{de-w}\\
\Theta \mE_4&=\mE_2\mE_4-e_2\mE_4.\label{dmq-w}
\end{align}}For an even integer $k \geq 2$, let $M_k(\Gamma_0(2))$ denote the space of modular forms of weight $k$ on $\Gamma_0(2).$ Then we know that the operator
\begin{equation}\label{operator1}
f \rightarrow \Theta f-\frac{k}{12}E_2f
\end{equation}
maps $M_k(\Gamma_0(2))$ to $M_{k+2}(\Gamma_0(2))$ (see Exercise no. 7 \cite[p. 123]{Koblitz}). Let $g(z):=E_2(z)-2E_2(2z).$ Then $g(z)\in M_2(\Gamma_0(2))$ (see \cite[Lemma 3.3]{BBY}). So for any constant $\alpha$, the operator
\begin{equation}\label{operator2}
f \rightarrow \Theta f-\frac{k}{12}(E_2+\alpha g)f
\end{equation}
maps $M_k(\Gamma_0(2))$ to $M_{k+2}(\Gamma_0(2)).$ In particular, if we set $\alpha=-2$, and use \eqref{3mp2e}, i.e., $\mE_2(z)=(4E_2(2z)-E_2(z))/3,$
then we have the following lemma.
\begin{lemma}\label{lemmaforoperator}
Let $f\in M_k(\Gamma_0(2)),$ then
\begin{equation}\label{myoperator}
\Theta f-\frac{k}{4}\mE_2 f \in M_{k+2}(\Gamma_0(2)).
\end{equation} 
\end{lemma}
Now, by \eqref{myoperator} applied to the modular form $e_2 \in M_2(\Gamma_0(2))$, we have $\Theta e_2-\frac{\mE_2e_2}{2} \in M_4(\Gamma_0(2)).$ By computing the first three terms in the $q$-expansion (which are enough to exceed the bound coming from the valence formula), we can prove the equality $$\Theta e_2-\frac{\mE_2e_2}{2}=-\frac{\mE_4}{2},$$ which is exactly \eqref{de-w}. Similarly, we can derive \eqref{dmq-w}, after applying \eqref{myoperator} to $\mE_4 \in M_4(\Gamma_0(2)).$

Since $\mE_2$ is not a modular form of weight $2$, we cannot use Lemma \ref{lemmaforoperator} to derive \eqref{dmp-w}. So first we prove that $\mE_2^2-4\Theta \mE_2 \in M_4(\Gamma_0(2)).$ We need the transformation formula for $\mE_2$.
\begin{lemma}
We have
\begin{equation}\label{rel-mE2}
\mE_2 \Big( \frac{az+b}{cz+d} \Big) =(cz+d)^2\mE_2(z)
+\frac{2}{\pi i}c(cz+d), \quad \begin{pmatrix} a
& b \\ c & d\end{pmatrix} \in \Gamma_0(2).
\end{equation}
\end{lemma}
\begin{proof}
Recall the transformation formula \cite[p. 68]{Schoeneberg} 
\begin{equation}\label{rel-E2}
E_2\Big( \frac{az+b}{cz+d} \Big) =(cz+d)^2E_2(z)+\frac{6}{\pi i}c(cz+d),
\quad \begin{pmatrix} a & b \\ c & d\end{pmatrix} \in SL_2(\mathbb{Z}).
\end{equation}
So for $\left(\begin{smallmatrix} a & b \\ c & d\end{smallmatrix}\right)\in \Gamma_0(2),$ 
$$E_2 \left(2\Big( \frac{az+b}{cz+d}\Big)\right)=E_2 \left(\frac{a(2z)+b}{\frac{c}{2}(2z)+d}\right)=(cz+d)^2E_2(2z)
+\frac{3}{\pi i}c(cz+d).$$ Hence,
{\allowdisplaybreaks\begin{align*}
\mE_2 \Big( \frac{az+b}{cz+d} \Big)=&\frac{4}{3}\Big((cz+d)^2E_2(2z)+\frac{3}{\pi i}c(cz+d)\Big)\\
&-\frac{1}{3}\Big((cz+d)^2E_2(z)+\frac{6}{\pi i}c(cz+d)\Big)\\
=&(cz+d)^2\mE_2(z)+\frac{2}{\pi i}c(cz+d).
\end{align*}}
\end{proof}
It is clear from \eqref{rel-mE2} that
\begin{equation}\label{thetamE2}
\Theta \mE_2 \Big( \frac{az+b}{cz+d} \Big)=(cz+d)^4\Theta \mE_2(z)+\frac{c(cz+d)^3}{\pi i}\mE_2(z)-\frac{c^2(cz+d)^2}{\pi^2}.
\end{equation}
So, by \eqref{rel-mE2} and \eqref{thetamE2}, for $\left(\begin{smallmatrix} a & b \\ c & d\end{smallmatrix}\right)\in \Gamma_0(2),$
{\allowdisplaybreaks\begin{align*}
\mE_2^2\Big(\frac{az+b}{cz+d}\Big)-4\Theta \mE_2\Big(\frac{az+b}{cz+d}\Big)=&\Big((cz+d)^2\mE_2(z)+\frac{2}{\pi i}c(cz+d)\Big)^2\\
-&4\Big((cz+d)^4\Theta \mE_2(z)+\frac{c(cz+d)^3}{\pi i}\mE_2(z)-\frac{c^2(cz+d)^2}{\pi^2}\Big)\\
=&(cz+d)^4\Big(\mE_2^2(z)-4\Theta \mE_2(z)\Big).
\end{align*}}
Hence $\mE_2^2-4\Theta \mE_2 \in M_4(\Gamma_0(2))$. Then by examining at the first three terms in the $q$-expansions, we find that $$\mE_2^2-4\Theta \mE_2=\mE_4,$$ which is the desired result \eqref{dmp-w}.

\section{A differential equation depending on weights of modular forms}

The differential equation in the upper half plane $z \in
\mathbb{H}$
\begin{equation}\label{fulldiff}
f''(z)-\frac{k+1}{6}E_2(z)f'(z)+\frac{k(k+1)}{12}E_2'(z)f(z)=0,
\end{equation}
was originally studied by M. Kaneko and D. Zagier in~\cite{K&Z}.
Here the symbol $'$ denotes the $\Theta$-operator $(2 \pi i)^{-1}
d/dz=q \cdot d/dq,~ (q=e^{2 \pi i z})$. For convenience, in this section, we will use this notation. Then it is known, that for $k \equiv 0, 4 \pmod{12} $, there exists a modular solution of~(\ref{fulldiff})
$$E_4(z)^{k/4}F\left( -\frac{k}{12}, -\frac{k-4}{12}; -\frac{k-5}{6}; \frac{1728}{j(z)} \right),$$ where
$$F(a, b; c; x)=\sum_{n=0}^{\infty}\frac{(a)_n(b)_n}{(c)_n}\frac{x^n}{n!} ,
\quad (a)_n=a(a+1)\cdots (a+n-1), \quad |x|<1,$$ is the Gaussian
hypergeometric series, and  $j(z)$ is the elliptic modular
invariant. Various modular forms on some subgroups were obtained
in \cite{K&K1} as solutions to this differential equation, where
the groups depend on the choice of $k$. In particular, on
$\Gamma_0(4)$, Ono~\cite{Ono} constructed a family of
differential endomorphisms and carried out a similar analysis for
modular forms on this group, including those of half-integral
weight.

In addition to the modular solutions, quite remarkable was an occurrence of a {\it quasi--modular form}, not of weight $k$ as in the modular case but of weight $k+1$. Along the same lines, Kaneko and Koike \cite{K&K2} found some examples of quasi--modular forms as solutions to an analogous differential equation attached to the group $\Gamma_0^*(2)$, which are {\it not} contained in the full modular group, where $\Gamma_0^*(2)$ is defined by
$$\Gamma_0^*(2)= \left\langle \Gamma_0(2), \begin{pmatrix} 0 & -1 \\ 2 & 0 \end{pmatrix}\right\rangle, $$
where the modular subgroup $\Gamma_0(2)$
is defined in \eqref{gamma2}.

The differential equation on which we focus in this section is
\begin{equation}\label{mydiff}
f''(z)-\frac{k+1}{2}\mathcal{E}_2(z)f'(z)+\frac{k(k+1)}{4}\mathcal{E}_2'(z)f(z)=0,
\end{equation}
where $\mE_2(z)$ is a quasi--modular form of weight $2$ on
$\Gamma_0(2)$ defined by \eqref{eisen}, i.e.,
$$\mathcal{E}_2(z)=1+8\sum_{n=1}^{\infty}\frac{(-1)^{n-1}nq^n}{1-q^n}.$$ In the present section,
for any positive even $k$, we construct solutions of the
differential equation \eqref{mydiff}, which are indeed modular
forms of weight $k$ on $\Gamma_0(2)$.

A simple calculation using the differential
equations~\eqref{dmp-w}--\eqref{dmq-w} shows that $\mathcal{E}_2$
is the logarithmic derivative of the modular form
\begin{equation}\label{D}
D:=\frac{e_2^2-\mathcal{E}_4}{64}=q+8q^2+28q^3+64q^4+ \cdots
\end{equation}
of weight $4$ on $\Gamma_0(2)$.

Define a sequence of polynomials $A_n(x)$ by
\begin{equation}\label{rr-A}
A_0(x)=1,~A_1(x)=x,~A_{n+1}(x)=xA_n(x)+\lambda_n A_{n-1}(x)~(n=1,2,\dots),
\end{equation} where
\begin{equation}\label{lambda}
\lambda_n=-64\frac{(n+1)^2}{(2n+1)(2n+3)}.
\end{equation}
The polynomial $A_n(x)$ is an even or odd polynomial according as
$n$ is even or odd, respectively. We also define a second sequence
of polynomials $B_n(x)$ by the same recursion with different
initial values as follows:
\begin{equation}\label{rr-B}
B_0(x)=0,~B_1(x)=1,~B_{n+1}(x)=xB_n(x)+\lambda_nB_{n-1}(x)~(n=1,2,\dots).
\end{equation}The polynomial $B_n(x)$ has
opposite parity, i.e., it is even if $n$ is odd and odd if $n$ is
even. Then our first result is given in the following theorem.

\begin{theorem}\label{ortho}
Let $k=2n+2~(n=0, 1, 2, \dots)$. Then the following modular form of
weight $k$ on $\Gamma_0(2)$, 
\begin{equation}\label{ortho-solution}
D^{n/2}A_n\Big(\frac{e_2}{\sqrt{D}}\Big)\frac{2e_2}{3}
+D^{(n-1)/2}B_n\Big(\frac{e_2}{\sqrt{D}}\Big)\frac{\mE_4}{3},
\end{equation}is a solution of \eqref{mydiff}, where $D$ is defined by \eqref{D}.
\end{theorem}

\noindent \textbf{Remark.} An element of degree $k$ in the ring $\mathbb{C} [e_2, D]$ is referred to as a modular form of weight $k$ on $\Gamma_0(2)$. Note
that $\sqrt{D}$ does not really occur due to the evenness and oddness of $A_n(x)$ and $B_n(x)$ on each $n$.

Let the operator $\vartheta_k$ be denoted by
\begin{equation}\label{theta}
\vartheta_k(f):=f'-\frac{k}{4}\mE_2f,
\end{equation} 
which is the formula \eqref{myoperator} in Lemma~\ref{lemmaforoperator}.
Using \eqref{de-w}, we find that
\begin{equation}
\vartheta_2(e_2)=-\frac{\mE_4}{2}.\label{theta-example1}
\end{equation}
Similarly by \eqref{theta} and \eqref{dmq-w}, we obtain
\begin{equation}\label{theta-example2}
\vartheta_4(\mathcal{E}_4)=-e_2\mE_4.
\end{equation}
We also deduce that
\begin{equation}\label{theta-example3}
\vartheta(D)=\frac{\big(\vartheta(e_2^2)-\vartheta(\mE_4)\big)}{64}=0,
\end{equation}after an application of \eqref{theta-example1} and
\eqref{theta-example2}.

If $f$ and $g$ have weights $k$ and $l$, the Leibniz rule
$$\vartheta_{k+l}(fg)=\vartheta_k(f)g+f\vartheta_l(g)$$ holds. 
We sometimes drop the suffix of the operator $\vartheta_k$ when the
weights of modular forms we consider are clear. With this
$\vartheta_k$ operator, the equation \eqref{mydiff} can be
rewritten in the following lemma.
\begin{lemma}
The differential equation \eqref{mydiff} is equivalent to
\begin{equation}\label{alt-mydiff}
\vartheta_{k+2}\vartheta_k(f)=\frac{k(k+2)}{16}\mE_4f.
\end{equation}
\end{lemma}
\begin{proof}
By the definition of the $\vartheta$--operator in \eqref{theta},
we obtain {\allowdisplaybreaks\begin{align*}
\vartheta_{k+2}\vartheta_k(f)&=\vartheta_{k+2}\Big(f'-\frac{k}{4}\mE_2f\Big)\\
&=\Big(f'-\frac{k}{4}\mE_2f\Big)'-\frac{k+2}{4}\mE_2\Big(f'-\frac{k}{4}\mE_2f\Big)\\
&=-\frac{k(k+2)}{4}\mE'_2f+\frac{k(k+2)}{16}\mE_2^2f,
\end{align*}}
where, in the last equality, we employed the equation \eqref{mydiff}, i.e.
$$f''-\frac{k+1}{2}\mE_2f'=-\frac{k(k+1)}{4}\mE'_2f.$$ So we
complete our proof by using \eqref{dmp}.
\end{proof}

\begin{proof}[Proof of Theorem \ref{ortho}.] Let $F_k$ denote the form in \eqref{ortho-solution} in Theorem~\ref{ortho}. We first establish the recurrence relation
\begin{equation}\label{rr}
F_{k+2}=e_2F_k+\lambda_n DF_{k-2},
\end{equation}
where $n=(k-2)/2$. This is a consequence of the recurrence
relations for $A_n$ and $B_n$ as in \eqref{rr-A} and \eqref{rr-B}, respectively. Then
{\allowdisplaybreaks\begin{align*}
e_2F_k+\lambda_nDF_{k-2}=&e_2\left(D^{n/2}A_n\left(\frac{e_2}{\sqrt{D}}\right)\frac{2e_2}{3}
+D^{(n-1)/2} B_n \left( \frac{e_2}{\sqrt{D}} \right)\frac{\mathcal{E}_4}{3} \right)\\
&+\lambda_n D \left( D^{(n-1)/2} A_{n-1} \left(
\frac{e_2}{\sqrt{D}}\right)\frac{2e_2}{3}
+D^{(n-2)/2} B_{n-1} \left( \frac{e_2}{\sqrt{D}} \right)\frac{\mathcal{E}_4}{3} \right)\\
=&D^{(n+1)/2} \left( \frac{e_2}{\sqrt{D}} A_n \left( \frac{e_2}{\sqrt{D}} \right) +\lambda_n A_{n-1}
\left( \frac{e_2}{\sqrt{D}} \right)\right)\frac{2e_2}{3}\\
&+D^{n/2} \left(\frac{e_2}{\sqrt{D}} B_n \left(
\frac{e_2}{\sqrt{D}} \right)+\lambda_n B_{n-1}
\left( \frac{e_2}{\sqrt{D}} \right) \right)\frac{\mathcal{E}_4}{3}\\
=&D^{(n+1)/2} A_{n+1} \left( \frac{e_2}{\sqrt{D}}
\right)\frac{2e_2}{3}
+D^{n/2} B_{n+1} \left( \frac{e_2}{\sqrt{D}} \right)\frac{\mathcal{E}_4}{3}\\
=&F_{k+2}.
\end{align*}}
Now we prove by induction that $F_k$ satisfies the equation
\eqref{alt-mydiff}. For the base step, we first have to check the
cases $k=2~(n=0)$ and $k=4~(n=1)$. In the case $k=2$, from the
formula \eqref{ortho-solution}, we have $F_2=2e_2/3,$ which is
clearly a modular form of weight $2$ on $\Gamma_0(2)$. Moreover,
using \eqref{theta-example1} and \eqref{theta-example2}, we obtain
$$\vartheta^2\left(\frac{2e_2}{3}\right)
=\vartheta \left(-\frac{\mathcal{E}_4}{3}\right)
=\frac{e_2\mathcal{E}_4}{3}=\frac{2(2+2)}{16}\mathcal{E}_4
\cdot\frac{2e_2}{3},$$ which satisfies the equation
\eqref{alt-mydiff}. Similarly, we can deduce that
$F_4=(2e_2^2+\mathcal{E}_4)/3$ is a modular form of weight $4$
satisfying \eqref{alt-mydiff}.

Assume $F_{k-2}$ and $F_k$ satisfy \eqref{alt-mydiff}. Then by
using \eqref{rr} and the formulas \eqref{theta-example1},
\eqref{theta-example2}, and \eqref{theta-example3}, we have
{\allowdisplaybreaks\begin{align}
\vartheta^2(F_{k+2})=&\vartheta\left(\vartheta(F_k)e_2-\frac{\mathcal{E}_4F_k}{2}\right)
+\lambda_nD\vartheta^2(F_{k-2})\nonumber\\
=&\vartheta^2(F_k)e_2+\frac{e_2\mathcal{E}_4F_k}{2}-\mathcal{E}_4\vartheta(F_k)
+\lambda_n D \vartheta^2(F_{k-2})\nonumber\\
=&\frac{k(k+2)}{16}e_2\mathcal{E}_4F_k+\frac{e_2\mathcal{E}_4F_k}{2}-\mathcal{E}_4\vartheta(F_k)
+\frac{k(k-2)}{16}\lambda_n D \mathcal{E}_4F_{k-2}\nonumber\\
=&\frac{k^2+2k+8}{16}e_2\mathcal{E}_4F_k+\frac{k(k-2)}{16}\lambda_nD\mathcal{E}_4F_{k-2}
-\mathcal{E}_4\vartheta(F_k). \label{alt-theta2}
\end{align}}
Hence we find that, by \eqref{alt-theta2} and \eqref{rr},
{\allowdisplaybreaks\begin{align*}
\vartheta^2(F_{k+2})-&\frac{(k+2)(k+4)}{16}\mathcal{E}_4F_{k+2}\\
=&\left( \frac{k^2+2k+8}{16}-\frac{(k+2)(k+4)}{16} \right)e_2\mathcal{E}_4F_k\\&+\left( \frac{k(k-2)}{16}-\frac{(k+2)(k+4)}{16} \right) \lambda_nD\mathcal{E}_4F_{k-2}-\mathcal{E}_4\vartheta(F_k)\\
=&-\mathcal{E}_4 \left(
\frac{k}{4}e_2F_k+\vartheta(F_k)+\frac{k+1}{2}\lambda_n D F_{k-2}
\right).
\end{align*}}
To prove the theorem it will therefore suffice to show that
\begin{equation}\label{final}
\frac{ke_2F_k}{4}+\vartheta(F_k)=-\frac{k+1}{2}\lambda_n D
F_{k-2}.
\end{equation}
Again, we will prove the equality \eqref{final} by induction on
$k$. For the case $k=4~(n=1)$, the equation is checked directly as
follows by \eqref{theta-example1} and \eqref{theta-example2}:
{\allowdisplaybreaks\begin{align*}
\frac{4e_2F_4}{4}+\vartheta(F_4)=&e_2\frac{2e_2^2+\mathcal{E}_4}{3}+
\vartheta \left( \frac{2e_2^2+\mathcal{E}_4}{3} \right)\\
=&(e_2^2-\mathcal{E}_4)\frac{2e_2}{3}\\
=&-\frac{5}{2} \left( -64\frac{(1+1)^2}{(2\cdot 1+1)(2\cdot 1+3)}
\right)
 \left( \frac{e_2^2-\mathcal{E}_4}{64} \right) \frac{2e_2}{3}\\
=&-\frac{4+1}{2}\lambda_1 D F_2,
\end{align*}}
where the last equality comes from the relation \eqref{lambda} and
\eqref{D}. Using \eqref{rr}, we can rewrite $F_{k+2}$ as
\begin{equation}\label{finalbase}
F_{k+2}=\frac{1}{2(k+1)}((k+2)e_2F_k -4\vartheta(F_k)).
\end{equation}

Assume that \eqref{final} is valid for $k$, i.e.,
$\vartheta^2(F_k)=0$. Hence by \eqref{rr} and by applying the
$\vartheta$--operator to $F_{k+2}$ in \eqref{finalbase}, we find
that {\allowdisplaybreaks\begin{align*}
\frac{k+2}{4}e_2F_{k+2}+\vartheta(F_{k+2})
=&\frac{k+2}{8(k+1)}e_2 ((k+2)e_2F_k-4\vartheta(F_k))\\
&+\frac{k+2}{2(k+1)}\left( -\frac{\mathcal{E}_4F_k}{2}+e_2
\vartheta(F_k) \right)-\frac{2}{k+1}\vartheta^2(F_k)\\
=&\frac{(k+2)^2}{8(k+1)}(e_2^2-\mathcal{E}_4)F_k\\
=&-\frac{k+3}{2}\lambda_{n+1}DF_k.
\end{align*}}
Here we have used the induction assumption. Hence the proof of \eqref{final} is complete, and so then the proof of Theorem \ref{ortho} is also complete.
\end{proof}

We next indicate that the solutions of \eqref{mydiff} have a hypergeometric structure. Let
$$j_2:=\frac{e_2^2}{D}=\frac{1}{q}+40+276q-2048q^2+\cdots.$$
Then $j_2$ is a $\Gamma_0(2)-$invariant function which generates
the field of modular functions on $\Gamma_0(2)$ and the normalized
function $j_2-40$ is often referred to as the ``Hauptmodul'' for the
group $\Gamma_0(2)$.

\begin{theorem}\label{hyper}
For even $k \geq 4$, the differential equation \eqref{mydiff} has
solutions which are normalized modular forms of weight $k$ on
$\Gamma_0(2)$, a generator of which is given by
\begin{equation}\label{hyper-solution}
e_2^{\frac{k}{2}}F\left( -\frac{k}{4}, -\frac{k-2}{4};
-\frac{k-1}{2}; \frac{64}{j_2}\right).
\end{equation}
\end{theorem}
\begin{proof}
It is sufficient to show that
$$f:=\sum_{\substack{0 \leq i\leq
k/4}}\frac{(-\frac{k}{4})_i(-\frac{k-2}{4})_i}{(-\frac{k-1}{2})_i
i!}64^i D^ie_2^{\frac{k}{2}-2i}$$ is a solution of
\eqref{alt-mydiff}, since $f$ is a normalized modular form of
weight $k$ on $\Gamma_0(2)$. Since $$\mathcal{E}_4=e_2^2-64D,$$ by
\eqref{theta}, we find that
$$\vartheta(e_2)=-\frac{\mathcal{E}_4}{2}=32D-\frac{e_2^2}{2}.$$
Using these, we obtain
$$\vartheta^2(D^ie_2^{\frac{k}{2}-2i})=\alpha D^ie_2^{\frac{k}{2}-2i+2}+\beta D^{i+1}e_2^{\frac{k}{2}-2i}
+\gamma D^{i+2}e_2^{\frac{k}{2}-2i-2}$$ with
\begin{equation}\label{abc}
\alpha=\frac{(k-4i)(k-4i+2)}{16},\quad \beta=-8(k-4i)^2, \quad
\gamma=256(k-4i)(k-4i-2).
\end{equation} Hence, for
$$f=\sum_{\substack{0 \leq i \leq k/4}}a_iD^ie_2^{\frac{k}{2}-2i} \quad \text{with}
\quad
a_i=64^i\frac{(-\frac{k}{4})_i(-\frac{k-2}{4})_i}{(-\frac{k-1}{2})_i
i!},$$ we have
$$\vartheta^2(f)-\frac{k(k+2)}{16}\mathcal{E}_4f=\sum_{\substack{0
\leq i \leq k/4}}a'_iD^{i+1}e_2^{\frac{k}{2}-2i},$$ for some
constants $a'_i$. We can complete the proof by showing that
$a'_i=0$. By the definition of $a_i$ and \eqref{abc}, we can
express $a'_i$ in terms of $a_i$, namely,
{\allowdisplaybreaks\begin{align*}
a'_i=&\frac{(k-4i-4)(k-4i-2)}{16}a_{i+1}-8(k-4i)^2a_i\\
     &+256(k-4i+4)(k-4i+2)a_{i-1}-\frac{k(k+2)}{16}(a_{i+1}-64a_i)\\
    =&a_i \times
   \Bigg\{ -\frac{(k-4i-4)(k-4i-2)(k-4i)(k-4i-2)}{2(k-2i-1)(i+1)}\\
   &\hspace{0.45 in} -8(k-4i)^2-32(k-2i+1)i\\
  &\hspace{0.45in} -\frac{k(k+1)}{16} \left(-8\frac{(k-4i)(k-4i-2)}{(k-2i-1)(i+1)}-64 \right)
  \Bigg\}\\=&0,
\end{align*}}
after a simple algebraic calculation.
\end{proof}

\noindent \textbf{Remark.} Solutions of \eqref{mydiff} can be reformulated in terms of Rankin--Cohen brackets \cite{Zagier}. For modular forms $f$ and $g$ of weights $k$ and $l$, define a modular form $[f, g]$ of weight $k+l+2$ by
$$[f, g]=kfg'-lf'g$$
(``Rankin--Cohen brackets of degree 1''). By the definition of the $\vartheta$--operator in \eqref{theta}, the above equation may also be written as
\begin{equation}\label{bracket}
[f, g]=kf\vartheta_l(g)-l\vartheta_k(f)g.
\end{equation}
\begin{lemma}
Suppose $F_k$ satisfies the differential equation
\eqref{alt-mydiff}. Then 
\begin{equation}\vartheta ([F_k, e_2])=\frac{k-2}{8}[F_k,
\mathcal{E}_4]\label{bracket-ex1}\end{equation} and
\begin{equation}\vartheta( (k-2)[F_k, \mathcal{E}_4])+4\vartheta([F_k,
e_2^2])=\frac{(k-4)(k+2)}{2}[F_k,
e_2].\label{bracket-ex2}\end{equation}
\end{lemma}
\begin{proof}
Since $F_k$ is a solution of \eqref{alt-mydiff}, $F_k$ satisfies
the equality
$$\vartheta^2(F_k)=\frac{k(k+2)}{16}\mathcal{E}_4F_k.$$  From this and
the use of \eqref{theta-example1} and \eqref{theta-example2}, we
obtain {\allowdisplaybreaks\begin{align*}
\vartheta ([F_k, e_2])=&\vartheta \left(-\frac{k}{2}F_k\mathcal{E}_4-2e_2\vartheta(F_k) \right)\\
=&-\frac{k}{2}\vartheta(F_k)\mathcal{E}_4+\frac{k}{2}e_2\mathcal{E}_4F_k-2\left( \frac{k(k+2)}{16}
\mathcal{E}_4F_k \right)e_2+\mathcal{E}_4\vartheta (F_k)\\
=&-\frac{k(k-2)}{8}e_2\mathcal{E}_4F_k-\frac{k-2}{2}\mathcal{E}_4
\vartheta(F_k)\\
=&\frac{k-2}{8}[F_k, \mathcal{E}_4],
\end{align*}} which proves \eqref{bracket-ex1}.
Similarly, by using \eqref{bracket}, \eqref{theta-example1} and
\eqref{theta-example2}, we can prove \eqref{bracket-ex2}.
\end{proof}

\section{A class of infinite series connected with triangular numbers}

On page $188$ of his lost notebook \cite{Ramanujan-lost},
Ramanujan examines the series,
\begin{equation}
T_{2k}(q):=1+\sum_{n=1}^{\infty}(-1)^n\{(6n-1)^{2k}q^{n(3n-1)/2}+(6n+1)^{2k}q^{n(3n+1)/2}
\}, \quad |q|<1.
\end{equation}
Note that the exponents $n(3n \pm 1)/2$ are the generalized
pentagonal numbers. The series $T_{2k}(q)$, $k=1, 2, \dots$, can
be represented in terms of the Eisenstein series $P(q)$, $Q(q)$,
and $R(q)$. The proofs of all the formulas on page $188$ are given
in~\cite{BY}.

If we define the series,
\begin{equation}\label{Triangular}
\mathcal{T}_{2k}:=\mathcal{T}_{2k}(q)=1+\sum_{n=1}^{\infty}(2n+1)^{2k}q^{n(n+1)/2},\quad
|q|<1,
\end{equation}
then we can obtain analogous formulas for $\mathcal{T}_{2k}$ in
terms of $e$, $\mP$, and $\mQ$. Observe that the exponents
$n(n+1)/2$ are the triangular numbers $T_n$ defined by 
$$T_n:=\frac{n(n+1)}{2}, \quad n\geq0.$$

Recall that Ramanujan's theta function $\psi(q)$ \cite[p.~36,
Entry 22 (ii)]{BCB3} is defined by
\begin{equation}\label{psi}
\psi(q)=\sum_{n=0}^{\infty}q^{n(n+1)/2}=\frac{(q^2;q^2)_{\infty}}{(q;q^2)_{\infty}},
\end{equation}
where $|q|<1$, and, for any complex number $a$, we write $(a;q)_{\infty}:=\prod_{n=1}^{\infty}(1-aq^{n-1}).$

We now state four formulas for $\mathcal{T}_{2k}$.
\begin{theorem}\label{T2k}
If $\mathcal{T}_{2k}$ is defined by \eqref{Triangular}, and $\mP$,
$e$, and $\mQ$ are defined by \eqref{mp}--\eqref{mq}, then\\
$\textup{(i)} \hspace{0.55in}\displaystyle \frac{\mathcal{T}_2(q)}{\psi(q)}=\mP,$\\
$\textup{(ii)}\hspace{0.5in}\displaystyle \frac{\mathcal{T}_4(q)}{\psi(q)}=3\mP^2-2\mQ,$\\
$\textup{(iii)}\hspace{0.5in}\displaystyle \frac{\mathcal{T}_6(q)}{\psi(q)}=15\mP^3-30\mP\mQ+16e\mQ,$\\
$\textup{(iv)}\hspace{0.5in}\displaystyle
\frac{\mathcal{T}_8(q)}{\psi(q)}=105\mP^4-420\mP^2\mQ+448e\mP\mQ-128e^2\mQ-4\mQ^2.$
\end{theorem}
\begin{proof}
Important in our proofs is the simple identity
\begin{equation}\label{basic2}
(2n+1)^2=8\frac{n(n+1)}{2}+1.
\end{equation}
Observe that, by \eqref{psi}, {\allowdisplaybreaks\begin{align*}
\mP=&1+8q\frac{d}{dq}\sum_{n=1}^{\infty}(-1)^n \log (1-q^n)\\
   =&1+8q\frac{d}{dq}\log \frac{(q^2;q^2)_{\infty}}{(q;q^2)_{\infty}}\\
   =&1+8q\frac{\frac{d}{dq}\psi(q)}{\psi(q)}.
\end{align*}}
Thus, using \eqref{basic2}, we find that
{\allowdisplaybreaks\begin{align}
\psi(q)\mP=&\psi(q)+8q\frac{d}{dq}
           \left( 1+\sum_{n=1}^{\infty}q^{n(n+1)/2} \right) \nonumber \\
   =&\psi(q)+8\sum_{n=1}^{\infty} \frac{n(n+1)}{2} q^{n(n+1)/2} \nonumber\\
   =&\psi(q)+\sum_{n=1}^{\infty}(2n+1)^2q^{n(n+1)/2}-\psi(q)+1 \nonumber\\
   =&\mathcal{T}_2(q).  \label{T2*}
\end{align}}
This completes the proof of (i).

In the proofs of the remaining identities of Theorem~\ref{T2k}, in
each case, we apply the operator
 $8q\frac{d}{dq}$ to the preceding identity. In each proof we also use the identities
\begin{equation}\label{recurT2k}
8q\frac{d}{dq}\mathcal{T}_{2k}(q)=\mathcal{T}_{2k+2}(q)-\mathcal{T}_{2k}(q),
\end{equation}
which follows from differentiation and the use of \eqref{basic2},
and
\begin{equation}\label{T2&psi}
8q\frac{d}{dq} \psi(q)=\mathcal{T}_2(q)-\psi(q),
\end{equation}
which arose in the proof of \eqref{T2*}.

We now prove (ii). Applying the operator $8q\frac{d}{dq}$ to
\eqref{T2*} and using \eqref{recurT2k} and \eqref{T2&psi}, we
deduce
that$$\mP(q)(\mathcal{T}_2-\psi(q))+\psi(q)8q\frac{d\mP(q)}{dq}=\mathcal{T}_4(q)-\mathcal{T}_2(q).$$
Employing (i) to simplify and using \eqref{dmp}, we arrive at
\begin{equation}\label{T4*}
\mathcal{T}_4(q)=(3\mP^2-2\mQ)\psi(q),
\end{equation}
as desired.

To prove (iii), we apply the operator $8q\frac{d}{dq}$ to
\eqref{T4*} and use \eqref{recurT2k} and \eqref{T2&psi} to deduce
that {\allowdisplaybreaks\begin{align*}
\mathcal{T}_6-\mathcal{T}_4=&8\Big(6\mP q\frac{d\mP}{dq}-2q\frac{d\mQ}{dq}\Big)\psi(q)+(3\mP^2-2\mQ)(\mathcal{T}_2-\psi(q))\\
=&\Big(12\mP(\mP^2-\mQ)-16(\mP\mQ-e\mQ)\Big)\psi(q)+(3\mP^2-2\mQ)(\mP-1)\psi(q),
\end{align*}}
where we used \eqref{dmp}, \eqref{dmq} and (i). If we now employ
\eqref{T4*} and simplify, we obtain (iii).

In general, by applying the operator $8q\frac{d}{dq}$ to
$\mathcal{T}_{2k}$ and using \eqref{recurT2k} and \eqref{T2&psi},
we find that
$$\mathcal{T}_{2k+2}-\mathcal{T}_{2k}=8q\frac{d}{dq}g_{2k}(\mP,e,\mQ)\psi(q)+\mP g_{2k}(\mP,e,\mQ),$$
where we define the polynomials $g_{2k}(\mP,e,\mQ),~k \geq 1$, by
\begin{equation}\label{g2k}
g_{2k}(\mP,e,\mQ):=\frac{\mathcal{T}_{2k}(q)}{\psi(q)}.
\end{equation}
Then proceeding by induction while using the formula \eqref{g2k}
for $\mathcal{T}_{2k}$, we find that
\begin{equation}\label{recg2k}
g_{2k+2}(\mP,e,\mQ)=8q\frac{d}{dq}g_{2k}(\mP,e,\mQ)+\mP
g_{2k}(\mP,e,\mQ).
\end{equation}
With the use of \eqref{recg2k} and the differential equations
\eqref{dmp}--\eqref{dmq}, it should now be clear how to prove the
remaining identity (iv), and so we omit further details.
\end{proof}

\noindent \textbf{Remark.} Observe from Theorem \ref{T2k} that a general formula for $g_{2k}(\mP,e,\mQ)$ contains all products $\mP^l e^m \mQ^n$, such that $2l+2m+4n=2k$. It seems to be extremely difficult to find a general formula for
$g_{2k}(\mP,e,\mQ)$ that would give explicit representations for
each coefficient of $\mP^l e^m \mQ^n$.

\section{A combinatorial identity}

Let $\wt_s$ be defined for $s, n \in \mathbb{N}$, by
\begin{equation}
\wt_s(n)=\sum_{d|n}(-1)^{d-1}d^s, \label{til}
\end{equation}
where $\wt_1(n)=\wt(n),$ and $\wt_s(n)=0$ if $n \notin \mathbb{N}$. Glaisher \cite{Glaisher-sum} defined seven quantities which depend on the divisors of $n$, including \eqref{til}, and found expressions for them in terms of the $\sigma_s(n)$ defined as \eqref{sigma}. For instance~\cite{Glaisher-sum},
\begin{equation}
\wt_s(n)=\sigma_s(n)-2^{s+1}\sigma_s(n/2). \label{wt-s}
\end{equation} Then the first formula (i) in Theorem \ref{T2k} has an interesting arithmetical interpretation.

\begin{theorem}\label{combinatorial}
Define $\wt(0)=\frac{1}{8}$. Then we have that
\begin{equation}\label{tsig_j}
8\sum_{\substack{ j+k(k+1)/2=n \\ j, k \geq 0}} \wt(j)=\left\{
\begin{array}{ll}
                   (2r+1)^2, &\text{if } n=r(r+1)/2,\\
                          0, &\text{otherwise}.
                    \end{array} \right.
\end{equation}
\end{theorem}
\begin{proof}
By expanding the summands of $\mP$ in \eqref{mp} in geometric
series and collecting the coefficients of $q^n$ for each positive
integer $n$, we find that
$$\mP(q)=1+8\sum_{n=1}^{\infty}\wt(n)q^n=8\sum_{n=0}^{\infty} \wt(n)q^n,$$ upon using the definition $\wt(0)=\frac{1}{8}.$ Thus, by \eqref{psi} and
Theorem~\ref{T2k}, (i) can be written in the form
\begin{equation}\label{coromain}
\left( 8\sum_{j=0}^{\infty} \wt(j)q^j \right) \cdot \left(\sum_{k=0}^{\infty}q^{k(k+1)/2}\right)=1+\sum_{n=1}^{\infty}(2n+1)^2q^{n(n+1)/2}.
\end{equation}
Equating coefficients of $q^n$, $n \geq 1$, on both sides of
\eqref{coromain}, we complete the proof.
\end{proof}

Let $\mathbb{N}$ be the set of positive integers. Define
\begin{equation}\label{A}
\mathcal{A}:=\{(x,y)\in\mathbb{N}^2 : 2x^2+y^2=8n+1, y\text{ is odd}, 2|x \}
\end{equation}
and
\begin{equation}\label{B}
\mathcal{B}:=\{(x,y)\in\mathbb{N}^2 : x^2+y^2=8n+1, y\text{ is odd}, 4|x \}.
\end{equation}Then we derive the following combinatorial corollary of Theorem \ref{combinatorial}.

\begin{corollary}
The number of elements of $\mathcal{A}$ and the number of elements of $\mathcal{B}$ have the same parity in all cases except when $n=r(r+1)/2$ and $r \equiv 1, 2 \pmod 4$.
\end{corollary}

\begin{proof}
 Since $$\wt(j)\equiv\left\{
\begin{array}{ll}
                    1 \pmod 2, &\textrm{if $j=m^2$ or $j=2m^2$},\\
                    0 \pmod 2, &\textrm{otherwise},
                    \end{array} \right. $$ we have that
\begin{equation}\label{AB}
\sum_{\substack{ j+k(k+1)/2=n \\ j, k \geq 0}} \wt(j)\equiv
\sum_{\substack{ j+k(k+1)/2=n \\ j \geq 1, k \geq 0\\j=m^2
\text{ or } j=2m^2}} \wt(j) \pmod 2.
\end{equation}
After changing variables, it is easy to see that $\mathcal{A}$ and $\mathcal{B}$ can be rewritten as 
$$\mathcal{A}=\{ (j, k)\,|\,j >0, k \geq 0, j+k(k+1)/2=n,
j=m^2~\}$$ and $$\mathcal{B}=\{\,(j, k)\,|\,j >0, k \geq 0,
j+k(k+1)/2=n, j=2m^2 \}.$$ Therefore, by \eqref{AB}, we find that
{\allowdisplaybreaks\begin{align*}
\# \mathcal{A} + \# \mathcal{B} &=\sum_{\substack{ j+k(k+1)/2=n \\
j\geq 1, k \geq 0\\j=m^2 \text{or} j=2m^2}} 1  \\
 &\equiv \sum_{\substack{ j+k(k+1)/2=n \\ j \geq 1, k \geq 0\\j=m^2
\text{or} j=2m^2}} \wt(j) \pmod 2 \\
&\equiv \sum_{\substack{ j+k(k+1)/2=n \\ j \geq 1, k \geq 0}}
\wt(j) \pmod 2 \\
&=\left\{ \begin{array}{ll}
                   \frac{(2r+1)^2-1}{8}, &\textrm{if $n=\frac{r(r+1)}{2}$},\\
                          0, &\textrm{otherwise}
                    \end{array} \right. \\
&\equiv \left\{ \begin{array}{ll}
                   1 \pmod 2, &\textrm{when $n=r(r+1)/2$ and $r \equiv 1,~2 \pmod 4$},\\
                   0 \pmod 2, &\textrm{otherwise}.
                    \end{array} \right.
\end{align*}}
So we conclude the result.
\end{proof}

\medskip
\textbf{Acknowledgements.} I am deeply indebted to Professors S. Ahlgren, B. C. Berndt, M. Boylan, and A. Zaharescu for their helpful comments and encouragement.



\begin{thebibliography}{99}

\bibitem{BCB3}
B. C. Berndt, \emph{Ramanujan's Notebooks, Part III}, Springer--Verlag, New York, 1991.

\bibitem{BBY}
B. C. Berndt, P. Bialek, and A. J. Yee, \emph{Formulas of Ramanujan for the power series coefficients of certain quotients of Eisenstein series}, IMRN, No. 21 (2002), 1077--1109.

\bibitem{BY}
B. C. Berndt and A. J. Yee, \emph{A page on Eisenstein series in Ramanujan's lost notebook}, Glasgow Math. J. \textbf{45} (2003), 123--129.

\bibitem{Boylan}
M. Boylan, \emph{Swinnerton--Dyer type congruences for certain Eisenstein series}, Contemporary Math. \textbf{291} (2001), 93--108.

\bibitem{Glaisher-sum}
J. W. L. Glaisher, \emph{On certain sums of products of quantities depending upon the divisors of a number}, Mess. Math. \textbf{15}
(1885), 1--20.

\bibitem{K&K1}
M. Kaneko and M. Koike, \emph{On modular forms arising from a differential equation of hypergeometric type}, Ramanujan J. \textbf{7} (2003), 145--164.

\bibitem{K&K2}
M. Kaneko and M. Koike, \emph{Quasimodular solutions of a differential equation of hypergeometric type}, in {\it Galois theory and modular forms}, 329--336, Dev. Math., 11, Kluwer Acad. Publ., Boston, 2004.

\bibitem{K&Z}
M. Kaneko and D. Zagier, \emph{Supersingular j-invariants, hypergeometric series, and Atkin's orthogonal polynomials}, AMS/IP Studies in Advanced Mathematics, \textbf{7} (1998), pp. 97--126.

\bibitem{Koblitz}
N. Koblitz, \emph{Introduction to ellpitic curves and modular forms}, Springer--Verlag, New York, 1993.

\bibitem{Ono}
K. Ono, \emph{Differential endomorphisms for modular forms on $\Gamma_0(4)$}, in {\it Symbolic Computation, Number Theory, Special Functions, Physics and Combinatorics}, F. G. Garvan and M. E. H. Ismail, eds., Kluwer, Dordrecht, 2001, pp. 223--229.

\bibitem{Ramamani-thesis}
V. Ramamani, \emph{Some Identities Conjectured by Srinivasa Ramanujan in His Lithographed Notes Connected with Partition Theory and Elliptic Modular Functions--Their Proofs--Inter
Connection with Various Other Topics in the Theory of Numbers and Some Generalizations}, Doctoral Thesis, University of Mysore, 1970.

\bibitem{Ramamani-paper}
V. Ramamani, \emph{On some algebraic identities connected with Ramanujan's work}, in {\it{Ramanujan International Symposium on Analysis}}, N. K. Thakare, ed., Macmillan India, Delhi, 1989, pp. 279--291.

\bibitem{Ramanujan-arith}
S. Ramanujan, \emph{On certain arithmetical functions}, Trans. Cambridge Philos. Soc. \textbf{22} (1916), 159--184.

\bibitem{Ramanujan-collect}
S. Ramanujan, \emph{Collected Papers}, Cambridge University Press, Cambridge, 1927, reprinted by Chesea, New York, 1962; reprinted by the American Mathematical Society, Providence, RI, 2000.

\bibitem{Ramanujan-lost}
S. Ramanujan, \emph{The Lost Notebook and Other Unpublished Papers}, Narosa, New Delhi, 1988.

\bibitem{Schoeneberg}
B. Schoeneberg, \emph{Elliptic modular functions}, Springer--Verlag, New York, Heidelberg, Berlin, 1970.

\bibitem{Zagier}
D. B. Zagier, \emph{Modular forms and differential operators}, Prod. Indian Acad. Sci. Math. Sci. \textbf{104} (1994), 57--75.

\end{thebibliography}
\end{document}